\documentclass[reqno]{amsart}

\usepackage{graphicx,amssymb,amsmath,amscd}

\theoremstyle{plain}
\newtheorem{theorem}{Theorem}[section]
\newtheorem{proposition}{Proposition}[section]

\theoremstyle{definition}
\newtheorem{definition}{Definition}[section]
\newtheorem*{definition*}{Definition}

\theoremstyle{remark}

\newtheorem*{remark*}{Remark}

\newtheorem*{notation*}{Notation}

\newcommand{\tx}{\tilde{x}}

\DeclareMathOperator{\interior}{int}

\newcommand{\parts}{\mathcal{P}}

\newcommand{\Nset}{\mathbb{N}}

\newcommand{\Rset}{\mathbb{R}}

\newcommand{\tf}{\tilde{f}}
\newcommand{\tth}{\tilde{h}}

\newcommand{\lb}[1]{{\underline{ #1}}}
\newcommand{\ub}[1]{{\overline{ #1}}}

\newcommand{\ix}{{\bf x}}

\newcommand{\id}{{\bf d}}

\newcommand{\tix}{\tilde{\bf x}}

\begin{document}

	\title{Manifold approximation of set-valued functions}
	\author{Alexandre Goldsztejn}
	\date{February 21st, 2005}

\maketitle


\section{Notations}

	Sequences are denoted by $x^{(k)}$ --- although $C^\infty$ functions are involved in the sequel, their derivative are not explicitly used so this notation for sequences will not interfere with the usual one for derivatives. A function $f:\Rset^0\longrightarrow\parts(\Rset^m)$ is a subset of $\Rset^m$. Notations and theorems related to manifolds are taken from Hirsch\cite{Hirsch1976}. Points belonging to manifolds will be denoted with tildes. The closed $n$-dimensional disk is denoted by $D^n$ and the $n$-dimensional sphere by $S^n$, so $\partial D^n=S^{n-1}$.

\section{Definition of ma-continuity for set-valued maps}

	Informally, a set-valued map is ma-continuous if it can be approximated by smooth manifolds. Up to now, the definition of ma-continuity is restricted to functions defined inside boxes.
	\begin{definition}
		Consider a set-valued function $F:\ix\subseteq\Rset^n\longrightarrow\parts(\Rset^m)$. If $n>0$, $F$ is {\it ma-continuous} if and only if there exists
		\begin{itemize}
			\item a closed disk $\id\subseteq\Rset^n$ such that $\ix\subseteq\interior\id$;
			\item $M^{(k)}$ a sequence of $C^\infty$ compact $n$-manifolds such that $\partial M^{(k)}=S^{n-1}$;
			\item $g^{(k)}:M^{(k)}\longrightarrow \Rset^{n+m}$ a sequence of $C^\infty$ maps such that $g^{(k)}|_{\partial M^{(k)}}$ is a $C^\infty$ diffeomorphism between $\partial M^{(k)}$ and $(\partial\id)\times\{0\}$.
		\end{itemize}
		such that any sequence $\tx^{(k)}\in M^{(k)}$, satisfies 
		\begin{itemize}
			\item the sequence $g^{(k)}(x^{(k)})$ is bounded;
			\item if furthermore $g^{(k)}(x^{(k)})\in \ix^{(k)}\times\Rset^m$ for all $k\in\Nset$, then any accumulation point $(x^*,y^*)^T$ of the sequence $g^{(k)}(x^{(k)})$ satisfies $y^*\in F(x^*)$;
		\end{itemize}
		If $n=0$, $f$ is ma-continuous if and only if $F\neq\emptyset$.
	\end{definition}
	The condition satisfied by the boundary of the manifolds was somewhat difficult to find and can certainly be improved. It has the advantage of simplifying the proof of the JOINT property although it may turn out to be difficult to prove the PROD property --- because Cartesian products of closed disks are not easily related to disks, are they ?

\section{JOINT}

	The following theorem is slightly more restrictive than the JOINT property as it asks for a strict inclusion instead of a simple inclusion. However, the JOINT property is certainly an easy consequence of this weaker property.
	\begin{theorem}
		Let $F:\ix\subseteq\Rset^n\longrightarrow \parts(\Rset^m)$, $m>0$, $n>0$, be a ma-continuous set-valued function. Given $i\in[1..m]$ and $j\in[1..n]$, suppose that there exists an interval $\tix\subseteq\interior\ix_j$ such that
		\begin{equation}\label{eq:hypothesis}
			\forall x\in\ix \ , \ F_i(x)\subseteq\tix
		\end{equation}
		where $F_i(x)$ stands for the projection of $f(x)$ onto the $i$-axis. Then the function $G:\ix_{\neq j}\subset\Rset^{n-1}\longrightarrow\parts(\Rset^m)$ defined by
		$$G(x_{\neq j})=\{y\in F(x) \ | \ y_i=x_j\}$$
		is also ma-continuous.
	\end{theorem}
	\begin{proof}[Sketch of the proof]
		First, suppose that $n>1$. Consider the $C^\infty$ function $\tth:\Rset\longrightarrow\ix_j$ defined like illustrated in the figure \ref{fig:th} --- it can be formally constructed using the function $\lambda$ defined in Hirsch\cite{Hirsch1976} page 42. As we can see, $\tth$ satisfies both
		\begin{equation}\label{eq:th-property-1}
			y\in\tix\Longrightarrow \tth(y)=y
		\end{equation}
		and
		\begin{equation}\label{eq:th-property-2}
			x=\tth(y)+\epsilon\wedge |\epsilon|\leq\alpha \ \Longrightarrow \ x\in\ix_j
		\end{equation}
		where $\alpha=\frac{1}{2}\:\min\{|\lb{\ix}_j-\lb{\tix}|,|\ub{\ix}_j-\ub{\tix}|\}$ --- we have $\alpha>0$ because by hypothesis $\tix\subseteq\interior\ix_j$. Now define the following functions.
		$$h:\Rset^{n+m}\longrightarrow \Rset \ \ ; \ \ h(x,y)=\tth(y_i)-x_j$$
		and
		$$h^{(k)}:M^{(k)}\longrightarrow\Rset \ \ ; \ \ h^{(k)}(\tx)=h(g^{(k)}(\tx))$$
		They are both $C^\infty$ because they are both composed of $C^\infty$ functions. The Morse-Sard theorem --- theorem 1.3 page 69 of Hirsch\cite{Hirsch1976} --- proves that critical values of $h^{(k)}$ and the critical values of $h^{(k)}|_{\partial M^{(k)}}$ have measure zero. So, we can pickup a sequence $\epsilon^{(k)}\in\Rset$, $0\leq |\epsilon^{(k)}|<\alpha$, which converges to zero such that $\epsilon^{(k)}$ is a regular value of both $h^{(k)}$ and $h^{(k)}|_{\partial M^{(k)}}$. Now define the sequence $N^{(k)}\subseteq M^{(k)}$ by
		$$N^{(k)}=\{\tx\in M^{(k)}|h^{(k)}(\tx)=\epsilon^{(k)}\}$$
		By the regular value theorem 4.1 page 31 of Hirsch\cite{Hirsch1976}\footnote{The dimension and compactness of the resulting manifolds are not explicitly written in the regular value theorem 4.1. However, they are easily proved to hold.}, $N^{(k)}$ is a compact $(n-1)$-submanifold of $M^{(k)}$ which satisfies $\partial N^{(k)}=N^{(k)}\cap \partial M^{(k)}$. We define $g^{\prime(k)}:N^{(k)}\longrightarrow\Rset^{n-1}\times\Rset^m$ by $g^{\prime(k)}(\tx)=(x_{\neq j},y)^T$ where $(x,y)^T=g^{(k)}(\tx)$. We now prove that the sequences $(N^{(k)})_{k\in\Nset}$ and $(g^{\prime(k)})_{k\in\Nset}$ satisfies the definition of ma-continuity for the function $G$.
		\\ \underline{First}, notice that $g^{\prime(k)}$ is a $C^{\infty}$ diffeomorphism between $\partial N^{(k)}$ and $\partial\id'\times\{0\}$ where $\id'\subseteq\Rset^{n-1}$ is a disk which contains $\ix_{\neq j}$ --- somehow because $\ix\subseteq\interior\id$ and by definition of $h$, we have $h(x,0)=x_j-c=0$ with $c\in\ix_j$, so the $(n-1)$-disk $\id'$ is the intersection of the $n$-disk $\id$ and the hyperplane $x_j-c=0$.
		\\ \underline{Second}, consider any sequence $\tx^{(k)}\in N^{(k)}$. We denote $g^{(k)}(\tx^{(k)})$ by $(x^{(k)},y^{(k)})^T$ so $g^{\prime(k)}(\tx^{(k)})=(x_{\neq j}^{(k)},y^{(k)})^T$. By definition of ma-continuous functions we know that the sequence $g^{(k)}(\tx^{(k)})$ is bounded so the sequence \underline{$g^{\prime(k)}(\tx^{(k)})$ is also bounded}. On the other hand, any accumulation point $(x^*_{\neq j},y^*)$ of the sequence $(x_{\neq j}^{(k)},y^{(k)})^T$ corresponds to at least one accumulation point $(x^*,y^*)$ of the sequence $(x^{(k)},y^{(k)})^T$. Consider any accumulation point $(x_{\neq j}^*,y^*)$ of $(x_{\neq j}^{(k)},y^{(k)})^T$ and one corresponding accumulation point $(x^*,y^*)$ of the sequence $(x^{(k)},y^{(k)})^T$. We suppose that $g^{\prime(k)}(\tx^{(k)})\in\ix_{\neq j}\times\Rset^m$ for all $k\in\Nset$ and we have to prove that $y^*\in G(x^*_{\neq j})$. We have $x^{(k)}_{\neq j}\in\ix_{\neq j}$. Now, by definition of $N^{(k)}$, we have $h^{(k)}(\tx^{(k)})=\epsilon^{(k)}$, that is $h(x^{(k)},y^{(k)})=\epsilon^{(k)}$, that is $\tth(y^{(k)}_i)=x^{(k)}_j+\epsilon_k$. This entails $\forall k\in\Nset,x^{(k)}_j\in\ix_j$ because $\epsilon_k< \alpha$ and thanks to the equation \eqref{eq:th-property-2}. Therefore, we have $x^{(k)}\in\ix$ and therefore
		$$\forall k\in\Nset, g^{(k)}(\tx^{(k)})\in\ix\times\Rset^{m}$$
		So, by definition of ma-continuity, the accumulation point $(x^*,y^*)$ of the sequence $(x^{(k)},y^{(k)})^T$ satisfies $x^*\in\ix$ because $\ix$ is closed, $y^*\in F(x^*)$ by definition of ma-continuity, and $y_i^*\in\tix$ by \eqref{eq:hypothesis}. Furthermore, because $h$ is continuous and because $h(x^{(k)},y^{(k)})=\epsilon^{(k)}$ for all $k\in\Nset$, the accumulation point $(x^*,y^*)$ also satisfies $h(x^*,y^*)=0$, that is $y^*_i=x^*_j$ because $y_i^*\in\tix$. Therefore, we have finally \underline{$y^*\in G(x^*_{\neq j})$}.
		\\ The case $n=1$ is different. It is a consequence of the fact that a compact $1$-manifold with boundary $S^0$, i.e. two points, is homeomorph to a segment. Then the proof should be achived thank to the intermediate value theorem.
	\end{proof}
		\begin{figure}[!t]
		\begin{center}
			\includegraphics[scale=0.35]{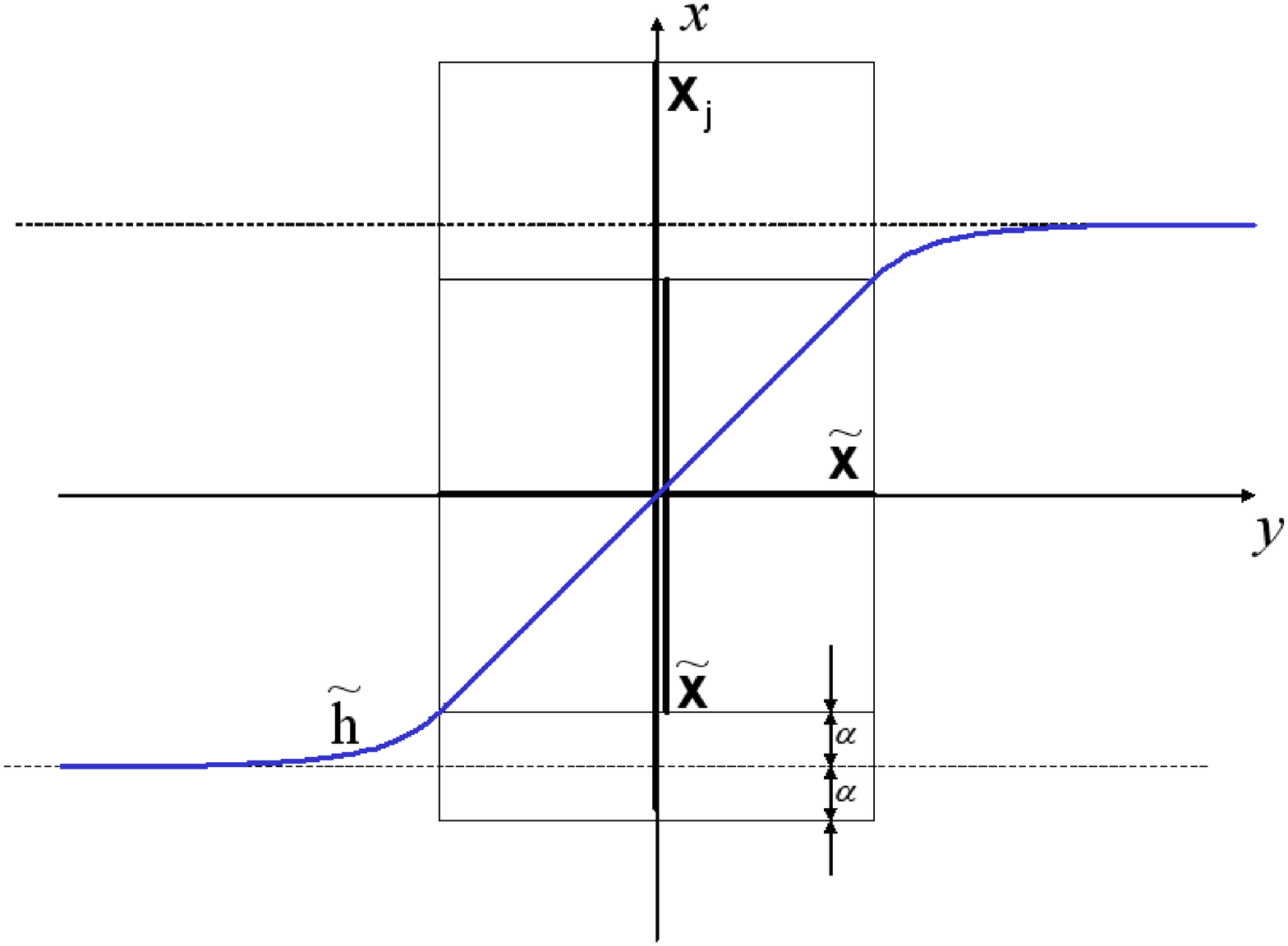}
			\caption{\label{fig:th}}
		\end{center}
	\end{figure}

\section{Continuous functions}

	Finally, the next proposition proves that continuous functions are ma-continuous.
	\begin{proposition}
		Let $f:\ix\subseteq\Rset^n\longrightarrow \Rset^m$ be a continuous function. Then, $F:\ix\subseteq\Rset^n\longrightarrow \Rset^m$ defined by $F(x)=\{f(x)\}$ is ma-continuous.
	\end{proposition}
	\begin{proof}[Sketch of the proof]
		Consider a disk $\id\subseteq\Rset^n$ such that $\ix\subseteq\interior\id$. Define the continuous function $\tf:\id\longrightarrow\Rset^m$ in such a way that $\tf|_\ix=f$ and $\tf(\partial\id)=0$ --- such a continuous function certainly exists. As $\tf|_{\partial\id}$ is constant equal to zero, it is $C^\infty$. Therefore, $\tf$ can be approximated by $C^\infty$ functions $\tf^{(k)}:\id\longrightarrow\Rset^m$ which satisfy $\tf^{(k)}(\partial\id)=0$ --- check Hirsch\cite{Hirsch1976} section 2 for approximation theorems. Then $M^{(k)}=D^n$ and $g^{(k)}=\tf^{(k)}\circ i$ --- where $i:D^n\longrightarrow \id\times\{0\}$ is the canonical $C^\infty$ embedding of $D^n$ into $\Rset^n\times\{0\}$  --- satisfy the hypothesis of the definition of ma-continuity.
	\end{proof}

\section{Conclusion}

	The property COMP seems to be impossible to achieved in a direct way. However, PROD and PROJ seems to be possible to achieve, and the COMP would then be a consequence of these latter.

\bibliographystyle{plain}

\begin{thebibliography}{1}

\bibitem{Hirsch1976}
M.W. Hirsch.
\newblock {\em Differential Topology}.
\newblock Springer-Verlag, 1976.

\end{thebibliography}

\end{document}